\documentclass[11pt]{amsart}
\usepackage{eurosym}
\usepackage[pagebackref,colorlinks=true]{hyperref}
\usepackage{verbatim}
\usepackage{eucal,url,amssymb,stmaryrd,enumerate,amscd,}
\usepackage{amsfonts}
\usepackage{amsmath,amsthm,amssymb,amscd,enumerate,eucal,url,stmaryrd}
\usepackage[margin=1in]{geometry}

\numberwithin{equation}{section}
\newtheorem{thrm}{Theorem}[section]
\newtheorem{lemma}[thrm]{Lemma}
\newtheorem{prop}[thrm]{Proposition}
\newtheorem{cor}[thrm]{Corollary}

\newtheorem{rmrk}[thrm]{Remark}

\allowdisplaybreaks

\begin{document}

\begin{abstract}
Conformally compact and complete   smooth
solutions to the Strominger system with non vanishing flux,
non-trivial instanton and non-constant dilaton using the first
Pontrjagin form of the \emph{$(-)$-connection} on 6-dimensional
non-K\"ahler nilmanifold are presented. In the conformally compact case the dilaton is determined by the
real slices of the elliptic Weierstrass function. The dilaton of
non-compact complete solutions is given by the fundamental
solution of the Laplacian on $R^4$.
\end{abstract}

\title[ Heterotic String Solutions with non-zero fluxes and
non-constant dilaton]{Non-Kaehler Heterotic String Solutions\\
with non-zero fluxes and non-constant dilaton}
\date{\today}
\author{Marisa Fern\'andez}
\address[Fern\'andez]{Universidad del Pa\'{\i}s Vasco\\
Facultad de Ciencia y Tecnolog\'{\i}a, Departamento de Matem\'aticas\\
Apartado 644, 48080 Bilbao\\
Spain}
\email{marisa.fernandez@ehu.es}
\author{Stefan Ivanov}
\address[Ivanov]{University of Sofia "St. Kl. Ohridski"\\
Faculty of Mathematics and Informatics\\
Blvd. James Bourchier 5\\
1164 Sofia, Bulgaria}
\address{and Institute of Mathematics and Informatics, Bulgarian Academy of
Sciences}
\email{ivanovsp@fmi.uni-sofia.bg}
\author{Luis Ugarte}
\address[Ugarte]{Departamento de Matem\'aticas\,-\,I.U.M.A.\\
Universidad de Zaragoza\\
Campus Plaza San Francisco\\
50009 Zaragoza, Spain}
\email{ugarte@unizar.es}
\author{Dimiter Vassilev}
\address[Dimiter Vassilev]{ Department of Mathematics and Statistics\\
University of New Mexico\\
Albuquerque, New Mexico, 87131-0001}
\email{vassilev@math.unm.edu}
\date{\today }
\maketitle
\tableofcontents



\setcounter{tocdepth}{2}

\section{Introduction}

The goal of this paper is the explicit construction of smooth
six-dimensional non-K\"ahler solution to the Strominger system with a
non-constant dilaton.

A model for string theory proposed in \cite{CHSW85} involves a ten
dimensional space $\mathbb{R}^{1,3}\times M^6$ which is the product of a
Lorentzian spacetime with a six-dimensional Calabi-Yau manifold
$M$. The latter was equipped with an SU(3) Yang-Mills connection,
Donaldson-Uhlenbeck-Yau instanton, of the Calabi-Yau metric. In a
key paper Strominger \cite{Str} systematically considered a
generalization of this construction allowing a background with
non-zero torsion, $H$-fluxes, which is motivated by physical
significance. This led to a system of differential equations known
as the \emph{Strominger system}, which specifies the geometric
{inner space} $M$ to be a complex (non-K\"ahler) conformally
balanced 6-manifold with holomorphically trivial canonical bundle
equipped, in addition, with an instanton bundle $E$ compatible
with the Green-Schwarz anomaly cancellation condition. The latter
also involves the first Pontrjagin form of a linear connection
whose determination is a part of the problem. An important problem
considered in the past thirty years is to provide ``backgrounds''
and
solutions of the Strominger system. 
Several connections have been used in order to satisfy the anomaly
condition, such as, the Levi-Civita connection \cite{Str,GMW}, the Chern
connection \cite{Str,y1,y2}, the $(+)$-connection \cite{CCDLMZ,DFG}, and the
$(-)$-connection \cite{Hull,Berg} etc..

A smooth compact solution was first found by Li \& Yau \cite{y1}
and Fu \& Yau \cite{y2,y3}. In \cite{y2}, developing the ideas of
\cite{GP}, the authors considered compact non-K\"ahler 6-manifolds
which are $\mathbb{T}^2$-bundles over a Calabi-Yau 4-manifold using the
first Pontrjagin form of the Chern connection. Thus, \cite{y2}
showed existence of a balanced metric while satisfying the
Hermitian-Yang-Mills equations and the anomaly equation. In this
case, the difficulty to satisfy anomaly cancellation condition
turns into a non-linear PDE of Monge-Amp\`ere type for the dilaton
function. We note that when $E$ is the tangent bundle of a
K\"ahler manifold $M$ the flux $H$
vanishes and Strominger's system is solved by the Calabi-Yau metric \cite%
{Yau} and the Donaldson-Uhlenbeck-Yau instanton \cite{U-Y,D}. In
particular, the non-K\"ahler case can be considered as a
generalization of Calabi's conjecture for the case of non-K\"ahler
Calabi-Yau threefolds. Since the choice of the first Pontrjagin
form of the $(-)$-connection is preferable by physical reasons
\cite{Berg}, in \cite{BSethi, BBCG} the authors considered the
smooth compact non-K\"ahler model of a $\mathbb{T}^2$-bundle over a
Calabi-Yau four manifold but with the first Pontrjagin form of the
$(-)$-connection in the anomaly cancellation. It was shown in
\cite{BSethi, BBCG} that the PDE system for the dilaton function
gives rise to a single PDE which is of the Laplace type for which
a solution exists provided the natural compatibility condition
holds.

Non-compact solutions can have different physical interpretation
in string theory \cite{CHS1,CHS,DLu,Str1}. They may be a local
models of a compact solutions or correspond to the supergravity
descriptions of the solitonic objects of the theory \cite{FLY}. A
class of non-compact smooth solutions to the Strominger system on
a $\mathbb{T}^2$ bundle over the non-compact Eguchi-Hanson space is
considered in \cite{FLY}  where the non-linear equation for the
dilaton imposed by the anomaly cancellation with the first
Pontrjagin form of the Chern connection is solved.

In this paper we construct  smooth
solutions with non vanishing flux and non-constant dilaton to the
Strominger system using the first Pontrjagin form of the
\emph{$(-)$-connection} on 6-dimensional complete non-compact manifold equipped with
conformally balanced Hermitian structures coupled with carefully
chosen instanton bundle. The source of the construction is the
already constructed smooth compact solutions to the Strominger
system with constant dilaton on nilmanifods presented in
\cite{FIUV} and the ideas of \cite{GMW} and \cite{GP}. In
particular, \cite{GMW} posed the question of solving the anomaly
condition for compact supersymmetric geometries that are two-torus
bundles over either conformally $\mathbb{T}^4$ or K3 manifold. Our main
results are explicit complete smooth examples of the
former case. 
In particular, we prove

\begin{thrm}
\label{t:cpmct strominger} The conformally compact manifold
$M^6=(\Gamma\backslash H_5, \bar g,J,\nabla^-, A_{\lambda})$
is a Hermitian manifold which solves the Strominger system with
non-constant dilaton $f$, non-trivial flux $H=\bar T$,
non-flat instanton $ A_{\lambda}$ using the first
Pontrjagin form of $\nabla^-$ and negative $\alpha^{\prime }$.
Furthermore, the heterotic equations of motion \eqref{mot} are
satisfied up to first order of $\alpha^{\prime }$.
\end{thrm}

The precise definition of the background and proof of Theorem
\ref{t:cpmct strominger} are given in Section \ref{ss:compact ex}.
Our solutions are complete non-K\"ahler
$\mathbb{T}^2$ bundles over conformally compact
asymptotically hyperbolic metric on $\mathbb{T}^4$ with conformal boundary
at infinity a flat torus $\mathbb{T}^3$. Using the first Pontrjagin form of
the $(-)$-connection together with the first Pontrjagin form of a
carefully chosen instanton we arrived via the anomaly cancellation
to a single highly non-linear PDE for the dilaton function.
Assuming that the dilaton depends only on one of the independent
variables we can reduce the equation to Weiestrass' equation. This
allows us to determine the dilaton as a real slice of an elliptic
function of order two, $f= \frac 12 \ln (\alpha^2 \mathcal{P})$
where $\mathcal{P}$ is the Weierstrass'
elliptic function with pole of order two at the origin, $z=\int^\mathcal{P}%
\frac {d\mathcal{P}}{2\sqrt {\mathcal{P}\left(
\mathcal{P}-a\right) \left( \mathcal{P}+a\right)}}$. The positive
parameter $a$ depends on the group $H_5 $ and the magnitude of
$\alpha^{\prime }$.  

This suggests there could be a relation with the F-theory/heterotic duality
principle. It is a well known fact that for warped compactification there
must be some branes with negative tension \cite{GKP}. It was argued in \cite%
{IMY,MY} (see also \cite{Sen0,HK,HP} for earlier discussions) that a
negative tension brane in heterotic string theory could be understood as a
T-dual of the Atiyah-Hitchin \cite{AH} manifold.

In Section \ref{ss:non-compact ex} we present another smooth
non-compact but complete solutions to the Strominger system using
the first Pontrjagin form of the $(-)$-connection with positive
string tension on certain $\mathbb{T}^2$ bundles over $\mathbb{R}^4$ with
non-vanishing torsion, non-trivial instanton and non-constant
dilaton. We construct an instanton whose first Pontrjagin form
together with the first Pontrjagin form of the $(-)$-connection
imposes via the anomaly cancellation a system of two equations of
Laplace type on the dilaton. The non-constant dilaton function of
our smooth non-compact complete solutions is determined by a
harmonic function, the fundamental solution of the Laplacian on
$\mathbb{R}^4$. In the special case of a trivial $\mathbb{T}^2$ bundle, i.e., on the
product of $\mathbb{T}^2\times \mathbb{R}^4$ we recover the 'symmetric background
solution' from \cite{CHS} thus strengthening the conjectured
existence of a non-compact non-trivial solution satisfying the
anomaly cancellation. The precise result is the following

\begin{thrm}
\label{t:main2} The non-compact simply connected manifold $(H_5, \bar
g,J,\nabla^-,A_{0,d})$ is a complete Hermitian manifold which solves the
Strominger system with non-constant dilaton $f$ determined by %
\eqref{fundsol}, non-zero flux $H=\bar T $ and non-flat
instanton $A_{0,d}$ using the first Pontrjagin form of
$\nabla^-$ and positive $\alpha^{\prime }$.

The complete  manifold $(H_5, \bar g,J,\nabla^-,A_{0,d})$
described above also solves the heterotic equations of motion \eqref{mot} up
to the first order of $\alpha^{\prime }$.
\end{thrm}

\textbf{Our conventions:} The connection 1-forms $\omega_{ji}$ of a metric
connection $\nabla, \nabla g=0$ with respect to a local orthonormal basis $%
\{E_1,\ldots,E_d\}$ are given by $\omega_{ji}(E_k) = g(\nabla_{E_k}E_j,E_i)$%
, since we write $\nabla_X E_j = \omega^s_j(X)\, E_s$.

The curvature 2-forms $\Omega^i_j$ of $\nabla$ are given in terms of the
connection 1-forms $\omega^i_j$ by \newline
\indent
$
\Omega^i_j = d \omega^i_j + \omega^i_k\wedge\omega^k_j, \quad \Omega_{ji} =
d \omega_{ji} + \omega_{ki}\wedge\omega_{jk}, \quad
R^l_{ijk}=\Omega^l_k(E_i,E_j), \quad R_{ijkl}=R^s_{ijk}g_{ls}. $%

\indent The first Pontrjagin class is represented by the 4-form $8\pi^2
p_1(\nabla)=\sum_{1\leq i<j\leq d} \Omega^i_j\wedge\Omega^i_j.$

\section{Motivation from heterotic string theory}

The bosonic fields of the ten-dimensional supergravity which arises as low
energy effective theory of the heterotic string are the spacetime metric $g$%
, the NS three-form field strength (flux) $H$, the dilaton $\phi$ and the
gauge connection $A$ with curvature 2-form $F^A$. The bosonic geometry is of
the form $\mathbb{R}^{1,9-d}\times M^d$, where the bosonic fields are non-trivial
only on $M^d$, $d\leq 8$. We consider the two connections $
\nabla^{\pm}=\nabla^g \pm \frac12 H, $ 
where $\nabla^g$ is the Levi-Civita connection of the Riemannian metric $g$.
Both connections preserve the metric, $\nabla^{\pm}g=0$ and have totally
skew-symmetric torsion $\pm H$, respectively. We denote by $R^g,R^{\pm}$ the
corresponding curvature.

We consider the heterotic supergravity theory with an $\alpha^{\prime }$
expansion where $1/2\pi\alpha^{\prime }$ is the heterotic string tension.
The bosonic part of the ten-dimensional supergravity action in the string
frame is (\cite{HT}, \cite{Berg}, $R=R^-$)
\begin{gather}  \label{action}
S=\frac{1}{2k^2}\int d^{10}x\sqrt{-g}e^{-2\phi}\Big[Scal^g+4(\nabla^g\phi)^2-%
\frac{1}{2}|H|^2 -\frac{\alpha^{\prime }}4\Big(Tr |F^A|^2)-Tr |R|^2\Big)\Big]%
.
\end{gather}
The string frame field equations (the equations of motion induced from the
action \eqref{action}) of the heterotic string up to the first order of $%
\alpha^{\prime }$ in sigma model perturbation theory are \cite{Hu86,HT} (we
use the notations in \cite{GPap})
\begin{gather}
Ric^g(X,Y)-\frac14<i_XH,i_YH>+2((\nabla^g)^2\phi)(X,Y)-\frac{\alpha^{\prime }%
}4\Big[<i_XF^A,i_YF^A>-<i_XR,i_YR>\Big]=0  \notag \\\label{mot}
\delta(e^{-2\phi}H)=-Tr(\nabla^g(e^{-2\phi}H))=0, \quad
\delta^{\nabla^+}(e^{-2\phi}F^A)=-Tr(\nabla^+(e^{-2\phi}F^A))=0,
\end{gather}
where $i_X$ is the interior multiplication of tensors and $<.,.>$ is the
corresponding scalar product.  The field equation of the dilaton $\phi$ is implied from the first two
equations above.

The Green-Schwarz anomaly cancellation mechanism requires that the
three-form Bianchi identity receives an $\alpha^{\prime }$ correction of the
form
\begin{equation}  \label{acgen}
dH=\frac{\alpha^{\prime }}48\pi^2(p_1(M^d)-p_1(E))=\frac{\alpha^{\prime }}4 %
\Big(Tr(R\wedge R)-Tr(F^A\wedge F^A)\Big),
\end{equation}
where $p_1(M^d)$ and $p_1(E)$ are the first Pontrjagin forms of $M^d$ with
respect to a connection $\nabla$ with curvature $R$ and the vector bundle $E$
with connection $A$, respectively.

A class of heterotic-string backgrounds for which the Bianchi identity of
the three-form $H$ receives a correction of type \eqref{acgen} are those
with (2,0) world-volume supersymmetry. Such models were considered in \cite%
{HuW}. The target-space geometry of (2,0)-supersymmetric sigma models has
been extensively investigated in \cite{HuW,Str,HP1}. Recently, there is
revived interest in these models \cite{Bwit,GKMW,CCDLMZ,GMPW,GMW,GPap} as
string backgrounds and in connection with heterotic-string compactifications
with fluxes \cite%
{Car1,BBDG,BBE,BBDP,y1,y2,y3,y4,P,GPR,GPRS,GLP,Pap,AG1,AG2,Bis}.

Equations \eqref{acgen}, \eqref{action} and \eqref{mot} involve a subtlety due to
the choice of the connection $\nabla$ on $M^d$ since anomalies can be canceled
independently of the choice \cite{Hull}. Different connections correspond to
different regularization schemes in the two-dimensional worldsheet
non-linear sigma model. Hence the background fields given for the particular
choice of $\nabla$ must be related to those for a different choice by a
field redefinition \cite{Sen}. Connections on $M^d$ proposed to investigate
the anomaly cancellation \eqref{acgen} are $\nabla^g$ \cite{Str,GMW}, $%
\nabla^+$ \cite{CCDLMZ,DFG,FIUV}, $\nabla^-$ \cite%
{Hull,Berg,Car1,GPap,II,KY,KM,MS,MY,IMY}, Chern connection $\nabla^c$ when $%
d=6$ \cite{Str,y1,y2,y3,y4}.

A heterotic geometry preserves supersymmetry if and only if, in 10
dimensions, there exists at least one Majorana-Weyl spinor $\epsilon$ such
that the following Killing-spinor equations hold \cite{Str,Berg}
\begin{equation}  \label{sup1}
\nabla^+\epsilon=0, \qquad (2d\phi-H)\cdot\epsilon=0, \qquad
F^A\cdot\epsilon=0,
\end{equation}
where $\cdot$ means Clifford action of forms on spinors. The system of Killing spinor equations \eqref{sup1} together with the
anomaly cancellation condition \eqref{acgen} is known as the \emph{%
Strominger system} \cite{Str,y1}.  The last equation in \eqref{sup1} is the instanton condition which means
that the curvature $F^A$ is contained in a Lie algebra of a Lie group which
is a stabilizer of a non-trivial spinor. In dimension 6 this group is $SU(3)$
and the last equation in \eqref{sup1} is the Donaldson-Uhlenbeck-Yau
instanton. The $SU(3)$-instanton means that the trace of $F^A$ with respect
to the K\"ahler 2 form as well as the (2,0)+(0,2)-part of $F^A$ vanish
simultaneously. The real expression of the $SU(3)$-instanton condition on a
six dimensional Hermitian manifold $(M,g,J)$ is given by
\begin{equation}  \label{in2}
(F^A)^i_j(JE_k,JE_l)=(F^A)^i_j(E_k,E_l),\qquad
\sum_{k=1}^6(F^A)^i_j(E_k,JE_k)=0.
\end{equation}

The first compact torsional solutions for the heterotic/type I string were
obtained via duality from M-theory compactifications on $\mathrm{K3}\times
\mathrm{K3}$ proposed in \cite{DRS}. The metric was first written down on
the orientifold limit in \cite{DRS} and such backgrounds have since been
studied (see \cite{BBDG,BBE} and references therein). The metric and the $H$%
-flux are derived by applying a chain of supergravity dualities and the
resulting geometry in the heterotic theory is a $\mathbb{T}^2$ bundle over a
K3.

Compact smooth examples in dimension six solving \eqref{sup1} and %
\eqref{acgen} with non-zero flux $H$ and non-constant dilaton were
constructed by Li and Yau \cite{y1} for U(4) and U(5) principal bundles
taking $R=R^c$-the curvature of the Chern connection in \eqref{acgen}.
Non-K\"ahler compact solutions of \eqref{sup1} and \eqref{acgen} on some
torus bundles over Calabi-Yau 4-manifold (K3 surfaces or complex torus) are
presented by Yau et al. \cite{y2,y3,y4} using the Chern connection in %
\eqref{acgen}. Compact solutions, up to two loops, in dimension six with
non-zero flux $H$ and non-constant dilaton involving the $(-)$-connection are
investigated in \cite{BSethi, BBCG}. Compact examples solving \eqref{sup1}
and \eqref{acgen} with nonzero field strength, non-trivial instanton,
constant dilaton and taking $R=R^+$, were constructed in \cite%
{CCDLMZ,DFG,FIUV}.


In the presence of a curvature term $Tr(R\wedge R)$ the solution of the
Strominger system \eqref{sup1}, \eqref{acgen} obey the second and the third
equations of motion (the second and the third equations in \eqref{mot}) but
do not always satisfy the Einstein equations of motion (see \cite{FIUV}
where a sufficient quadratic condition on $R$ is found). It was proved in
\cite{Iv0} that \eqref{sup1} and \eqref{acgen} imply \eqref{mot} if and only
if $R$ is an instanton in dimensions 5,6,7,8, (see \cite{MS} for higher
dimensions). In particular, in dimension 6, $R$ is required to be an SU($3$%
)-instanton.

The physically relevant connection on the tangent bundle to be considered in %
\eqref{acgen}, \eqref{action}, \eqref{mot} is the $(-)$-connection \cite%
{Berg,Hull}. One reason is that the curvature $R^-$ of the $(-)$-connection is
an instanton up to the first order of $\alpha^{\prime }$ which is a
consequence of the first equation in \eqref{sup1}, \eqref{acgen} and the
well known identity
\begin{equation}  \label{dtr}
R^+(X,Y,Z,U)-R^-(Z,U,X,Y)=\frac12dH(X,Y,Z,U).
\end{equation}
Indeed, \eqref{acgen} together with \eqref{dtr} imply $%
R^+(X,Y,Z,U)-R^-(Z,U,X,Y)=O(\alpha^{\prime })$ and the first equation in %
\eqref{sup1} yields that the holonomy group of $\nabla^+$ is contained in $%
SU(n)$, i.e. the curvature 2-form $R^+(X,Y)\subset \mathfrak{su}(n)$ and therefore $R^-$
satisfies the instanton condition \eqref{in2} up to the first order of $%
\alpha^{\prime }$. Hence, a solution to the Strominger system with first
Pontrjagin form of the $(-)$-connection always satisfies the heterotic
equations of motion \eqref{mot} up to the first order of $\alpha^{\prime }$
(see e.g.\cite{MS} and references therein).

We remark that in the case of compact Hermitian manifold with holomorhically
trivial canonical bundle, the vanishing theorem from \cite{IP1,IP2} shows
that $R^-$ is an instanton if and only if the manifold is K\"ahler. Indeed, %
\eqref{dtr} yields that if $R^-$ is an instanton then the trace of $dH$ with
respect to the K\"ahler form vanishes since the holonomy group of $\nabla^+$
is contained in $\mathfrak{su}(3)$. Hence, the function $h$ defined in \cite{IP1,IP2}
as the trace of $dH$ vanishes which implies, due to \cite[Corollary~4.2]{IP1},
that there are no holomorphic top-forms unless the manifold is K\"ahler.

Concerning the Chern connection, it is shown in \cite{MS} that the curvature
of the Chern connection is an instanton up to zeros order of $\alpha^{\prime
}$ if and only if the $H$-flux vanishes and the manifold is K\"ahler. The
proof in \cite{MS} relies on a point-wise identity established in \cite{IP1}
and therefore the result is purely local.

\subsection{The geometric model}

\label{geomod}

Necessary and sufficient conditions to have a solution to the system of
gravitino and dilatino equations (the first two equations in \eqref{sup1})
in dimension $2n$ were derived by Strominger in \cite{Str} involving the
notion of $SU(n)$-structure and then studied by many authors \cite%
{GKMW,GMPW,GMW,CCDLMZ,Car1,II,BBDG,BBE,GPap,y1,y2,y3,y4}.

The gravitino equation, the first equation in \eqref{sup1} shows that there
exists a parallel spinor with respect to the $(+)$-connection. This reduces
the structure group $SO(2n)$ to a subgroup of $SU(n)$ since the holonomy
group of $\nabla^+$ reduces to a subgroup of $SU(n)$, i.e., the manifold is
an almost Hermitian manifold admitting a linear connection having totally
skew-symmetric torsion which preserves both the almost Hermitian structure
and a non-vanishing $(n,0)$-form (complex volume form).

The dilatino equation, the second identity in \eqref{sup1}, yields that the
almost complex structure is integrable and the trace of the torsion 3-form
with respect to the K\"ahler form is an exact 1-form. Strominger shows in
\cite{Str} the existence of a unique Hermitian connection with
skew-symmetric torsion on any Hermitian manifolds writing explicitly the
torsion 3-form from the exterior derivative of the K\"ahler form ($\nabla^+$
in our notations). He also shows that the $\nabla^+$-parallel complex volume
form supplies a holomorphic complex volume form whose norm determines the
dilaton.

Next we detail the model in dimension six which is the focus of the paper.
Let $(M,J,g)$ be a Hermitian 6-manifold with Riemannian metric $g$ and a
complex structure $J$. The Kaehler form $F$ and the Lee form $\theta$ are
defined by $F(\cdot,\cdot)=g(\cdot,J\cdot), \quad \theta(\cdot)=\delta
F(J\cdot),$ respectively, where $*$ is the Hodge operator and $\delta$ is
the co-differential, $\delta=-*d*$. The flux $H$, i.e., the torsion
of the connection $\nabla^+$ preserving the
Hermitian structure $(J,g)$ is given by \cite{Str}
\begin{equation}  \label{tor0}
H=T=d^cF, \qquad \text{where} \qquad d^cF(X,Y,Z)=-dF(JX,JY,JZ).
\end{equation}
Clearly, \eqref{tor0} determines the connection $\nabla^+$ uniquely since $%
\nabla^+g=0$.

An SU($3$)-structure is determined by an additional non-degenerate
(3,0)-form $\Psi=\Psi^+ + \sqrt{-1}\, \Psi^-$, or equivalently by a
non-trivial spinor, satisfying the compatibility conditions $%
F\wedge\Psi^{\pm}=0$, $\Psi^+\wedge\Psi^-=\frac23F\wedge F\wedge F$.
The subgroup of $SO(6)$ fixing the forms $F$ and $\Psi$ simultaneously is SU(%
$3$). The Lie algebra of SU($3$) is denoted $\mathfrak{s}\mathfrak{u}(3)$.

The necessary and sufficient condition for the existence of solutions to the
first two equations in \eqref{sup1} derived by Strominger \cite{Str} imply
that the 6-manifold should be a complex conformally balanced manifold (the
Lee form $\theta=2d\phi$) with non-vanishing holomorphic volume form $\Psi$
satisfying an additional condition. In terms of the five torsion classes on
dimension six, described in \cite{CS}, the Strominger condition is
interpreted in \cite{CCDLMZ} as follows (see \cite{II} for a slightly
different expression):
\begin{equation}  \label{strsys}
2F\lrcorner dF+\Psi^+\lrcorner d\Psi^+=0,
\end{equation}
where $\lrcorner$ denotes the interior multiplication. Another very useful
interpretation of this condition was proposed in \cite{y1}. If the dilaton
is constant (the Lee form $\theta =0$) then the Strominger condition reads
\begin{equation}  \label{bal}
dF\wedge F=d\Psi^+=d\Psi^-=0.
\end{equation}
Compact examples of the latter on nilmanifolds were presented in \cite%
{UV2,UV1} and examples via evolution equations were given in \cite{FTUV}.

A very promising geometric model in dimension six was proposed by
Goldstein and Prokushkin in \cite{GP} as a certain
${\mathbb{T}}^2$-bundle over a Calabi-Yau surface, which we
explain next.
Let $\Gamma_i$, $1\leq i \leq 2$, be two closed $2$-forms on a Calabi-Yau
surface $M^4$ with anti-self-dual (1,1)-part, which represent integral
cohomology classes. Denote by $\omega_1$ and by $\omega_2+\sqrt{-1}\omega_3$
the (closed) K\"ahler form and the holomorphic volume form on $M^4$,
respectively. Then, there is a (non-K\"ahler)  6-dimensional manifold
$M^6$, which is the total space of a ${\mathbb{T}}^2$-bundle over $M^4$, and
it has an $SU(3)$-structure
\begin{equation}  \label{g2def}
g=g_{CY}+\eta_1^2+\eta_2^2, \quad F=\omega_1+\eta_1\wedge\eta_2, \quad
\Psi^+=\omega_2\wedge\eta_1-\omega_3\wedge\eta_2, \quad
\Psi^-=\omega_2\wedge\eta_2+\omega_3\wedge\eta_1,
\end{equation}
where $\eta_i$, $1\leq i \leq 2$, is a $1$-form on $M^6$ such that $%
d\eta_i=\Gamma_i$, $1\leq i \leq 2$. From the construction it is easy to
check that the $SU(3)$ structure \eqref{g2def} satisfies \eqref{bal} and
therefore it solves the first two Killing spinor equations in \eqref{sup1}
with constant dilaton.

For any smooth function $f$ on $M^4$, the $SU(3)$-structure on $M^6$ given
by
\begin{equation*} 
F=e^{2f}\omega_1+\eta_1\wedge\eta_2, \quad \Psi^+=e^{2f}\Big[%
\omega_2\wedge\eta_1-\omega_3\wedge\eta_2\Big], \quad \Psi^-=e^{2f}\Big[%
\omega_2\wedge\eta_2+\omega_3\wedge\eta_1\Big]
\end{equation*}
satisfies \eqref{strsys} and therefore it solves the first two Killing spinor
equations in \eqref{sup1} with non-constant dilaton $\phi=2f$. The metric
has the form
\begin{equation*}
g_f=e^{2f}g_{cy}+\eta_1^2+\eta_2^2.
\end{equation*}
This ansatz guaranties solution to the first two equations in  \eqref{sup1}.
To achieve a smooth solution to the Strominger system we still
have to determine an auxiliary vector bundle with an instanton and a
linear connection on $M^6$ 
in order to satisfy the anomaly cancellation condition \eqref{acgen}%
. Taking the first Pontrjagin form of the Chern connection \cite%
{y1,y2,y3,y4}  leads to an equation of Monge-Amp\`ere type for the dilaton
function, while it is reduced to a PDE of Laplace type for the dilaton when
using the first Pontrjagin form of the $(-)$-connection \cite{BSethi,BBCG}.

The ${\mathbb{T}}^2$-bundle over a K3 surface construction with connection
1-forms of anti-self-dual curvature was used in \cite{y1,y2,y3,y4} to produce
the first compact smooth solutions in dimension 6 solving the heterotic
supersymmetry equations \eqref{sup1} with non-zero flux and non-constant
dilaton together with the anomaly cancellation \eqref{acgen} with the first
Pontrjagin form of the Chern connection.

\section{The anomaly cancellation and the non-constant dilaton}

We apply the construction from Section~\ref{geomod} to special non-K\"ahler
2-step nilmanifolds which are $\mathbb{T}^2$-bundles over $\mathbb{T}^4$
with connection 1-forms of anti-self-dual curvature  on the four torus and
using the first Pontrjagin form of the $(-)$-connection in investigating the
anomaly cancellation \eqref{acgen} with non-constant dilaton.

\subsection{Two-step nilmanifolds with Abelian complex structure}

In this subsection we show, due to considerations in \cite{UV2}, that the
2-step nilmanifolds which are $\mathbb{T}^2$ bundles over $\mathbb{T}^4$
with connection 1-forms of anti-self-dual curvature are precisely the
balanced Hermitian structures with Abelian complex structure, i.e. $%
[JX,JY]=[X,Y]$.

{\ The invariant balanced Hermitian structures on compact 6-dimensional
nilmanifolds which are a $\mathbb{T}^2$-bundle over a 4-torus, according to
\cite[Theorem 2.11]{UV2}, are parametrized by one of the following three
sets of equations
\begin{equation}  \label{h5real}
de^1 = de^2=de^3=de^4=0,\quad de^5 = t\,(e^{13} -e^{24}) ,\quad de^6 =
t\,(e^{14} +e^{23}),
\end{equation}
where $t\in\mathbb{R}^*$; 
\begin{equation}  \label{fam1}
\begin{cases}
\begin{array}{lcl}
de^1 {\!\!\!} & = & {\!\!\!} de^2=de^3=de^4=0, \\[4pt]
de^5 {\!\!\!} & = & {\!\!\!} \frac{t}{s}(\rho+b^2)e^{13} -\frac{t}{s}%
(\rho-b^2) e^{24} , \\[5pt]
de^6 {\!\!\!} & = & {\!\!\!} -2\,t\,(e^{12}-e^{34}) + \frac{t}{s}%
(\rho-b^2)e^{14} + \frac{t}{s}(\rho+b^2)e^{23},%
\end{array}%
\end{cases}%
\end{equation}
where $\rho\in\{0,1\}$, $b\in \mathbb{R}$ and $s,t\in\mathbb{R}^*$; 
\begin{equation}  \label{fam2}
\begin{cases}
\begin{array}{lcl}
de^1{\!\!\!} & = & {\!\!\!} de^2=de^3=de^4=0, \\[6pt]
de^5{\!\!\!} & = & {\!\!\!} s Y \left[ 2 b^2 u_1 |u|\,(e^{12}-e^{34}) - b^2
t u_1 |u| Y\,(e^{13} + e^{24}) + 2 \rho s u_1\,(e^{13} - e^{24}) \right. \\%
[4pt]
&  & \quad \left. + 2 s u_2\left((\rho-b^2)e^{14}+(\rho+b^2)e^{23}\right)
\right], \\[6pt]
de^6 {\!\!\!} & = & {\!\!\!} s Y \left[ 2 (2s^2- b^2 u_2)
|u|\,(e^{12}-e^{34})+ b^2 t u_2 |u| Y\,(e^{13}+e^{24}) - 2 \rho s
u_2\,(e^{13}-e^{24})\right. \\[4pt]
&  & \quad \left. + 2 s u_1\left((\rho-b^2)e^{14}+(\rho+b^2)e^{23}\right)
\right],%
\end{array}%
\end{cases}%
\end{equation}
where $\rho\in\{0,1\}$, $b\in \mathbb{R}$, $t\in\mathbb{R}^*$ and $u\in
\mathbb{C}^*$ such that $s^2>|u|^2>0$, and where $Y=\frac{2 \sqrt{s^2-|u|^2}%
}{|u|t}$. }

{\ In all the cases the balanced structure $(J,F)$ is given in the standard
form, i.e.
\begin{equation}  \label{adapted-basis}
Je^1=-e^2,\ Je^3=-e^4,\ Je^5=-e^6,\quad\quad F=e^{12}+e^{34}+e^{56}.
\end{equation}
}

{\ The nilpotent Lie algebras (nilmanifolds) underlying the
    families \eqref{h5real}--\eqref{fam2} are $\mathfrak{h}_k$, $2\leq k\leq 6$
    (see \cite{UV2} for a description) }
However, in order to apply the construction from Section~\ref{geomod} we are led to

\begin{lemma}
\label{lem}
Let $(J,F)$ be an invariant balanced Hermitian structure on a $6$-dimensional $2$-step nilmanifold $M$.
Then, $M$ is the total space of a $\mathbb{T}^2$-bundle over $\mathbb{T}^4$ of anti-self-dual curvature
if and only if the complex structure $J$ is Abelian. Moreover, in such case the Lie algebra underlying $M$ is isomorphic to
$\mathfrak{h}_3$ or $\mathfrak{h}_5$.
\end{lemma}

\begin{proof}
The curvature of the bundle is determined by the 2-forms $de^5$ and $de^6$
in the structure equations \eqref{h5real}, \eqref{fam1} and \eqref{fam2}. Taking into account that $s,t,u,Y\not=0$
we get that $de^5, de^6 \in \langle e^{12}-e^{34},e^{13}+e^{24},e^{14}-e^{23}\rangle$
if and only if the balanced Hermitian structure is given by \eqref{fam1} or \eqref{fam2} with $\rho=0$.
The latter condition means that $J$ is an Abelian complex structure. Finally,
when $\rho=0$ we can take $b\in\{0,1\}$ since the corresponding balanced Hermitian structures are
isomorphic. The case $b=0$, resp. $b=1$, corresponds to
structures on the Lie algebra $\mathfrak{h}_3$, resp. $\mathfrak{h}_5$.
\end{proof}

{\
Notice that $\mathfrak{h}_3$ is the Lie algebra underlying the nilmanifold given by the
product of the 5-dimensional generalized Heisenberg nilmanifold by $S^1$,
whereas $\mathfrak{h}_5$ is the Lie algebra underlying the Iwasawa manifold.
It is important to note that the holonomy of the $(+)$-connection of any balanced structure $(J,F)$
with $J$ Abelian is a subgroup of SU(2), hence inside SU(3), \cite{UV2}.}


\subsection{Non-constant dilaton in 6-D}

\label{s:alg5}

Here we consider the Lie algebra $\mathfrak{h}_5$, which we describe below.
We shall construct a background with non-constant dilaton with non-trivial
instanton and flux. By a contraction, this will also give analogous
solutions on the Lie algebra $\mathfrak{h}_3$ as we shall explain later in
Section~\ref{s:alg all 6-D}.

The structure equations of the Lie algebra $\mathfrak{h}_5$ are
\begin{equation}  \label{fam1h5}
de^1 = de^2=de^3=de^4=0,\quad de^5 = \frac{t}{s}\,(e^{13} + e^{24}),\quad
de^6 = -2\, t\,(e^{12}-e^{34}) - \frac{t}{s}\, (e^{14}-e^{23}),
\end{equation}
where $s,t\in\mathbb{R}^*$. We note that the structure equations %
\eqref{fam1h5} 
are obtained from the family \eqref{fam1}
taking there $\rho=0$ and $b=1$. The corresponding Lie group $H_5$ can be
considered as a $\mathbb{R}^2$-bundle over $\mathbb{R}^4$. Moreover, the
balanced structure $(J,F)$ on $\mathfrak{h}_5$ is given in the standard form
given by \eqref{adapted-basis}.


Let $f$ be a smooth function on $\mathbb{R}^4$. Following \cite{GP} we
consider the metric $\bar{g}$ on $\mathfrak{h}_5$ for which the basis of
1-forms
\begin{equation}  \label{conf}
\bar{e}^{1}=e^{f}\,e^{1},\quad \bar{e}^{2}=e^{f}\,e^{2},\quad \bar{e}%
^{3}=e^{f}\,e^{3},\quad \bar{e}^{4}=e^{f}\,e^{4},\quad \bar{e}%
^{5}=e^{5},\quad \bar{e}^{6}=e^{6}
\end{equation}
is orthonormal.
The K\"ahler form of the new Hermitian structure $(\bar g,J)$ is given by
\begin{equation*}
\bar{F}=\bar{e}^{12}+\bar{e}^{34}+\bar{e}^{56}=e^{2f}(e^{12}+e^{34})+e^{56},
\end{equation*}
where $df=\sum_{i=1}^{4}f_{i}e^{i}$, i.e., in local coordinates $f_{i}=\frac{%
\partial f}{\partial x_{i}}$. Furthermore,
\begin{equation*}
\begin{array}{rcl}
d\bar{F}\!\! & \!\!=\!\! & \!\!2e^{-f}f_{3}\,\bar{e}^{123}+2e^{-f}f_{4}\,%
\bar{e}^{124}+2te^{-2f}\bar{e}^{125}+2e^{-f}f_{1}\,\bar{e}^{134}+\frac{t}{s}%
e^{-2f}\bar{e}^{136}+\frac{t}{s}e^{-2f}\bar{e}^{145} \\[7pt]
\!\! & \!\!+\!\! & \!\!2e^{-f}f_{2}\,\bar{e}^{234}-\frac{t}{s}e^{-2f}\bar{e}%
^{235}+\frac{t}{s}e^{-2f}\bar{e}^{246}-2te^{-2f}\bar{e}^{345}.%
\end{array}%
\end{equation*}%
According to \eqref{tor0}, the torsion 3-form $\bar{T}$ is represented by
\begin{equation}  \label{tor5}
\begin{array}{rcl}
\bar{T}=Jd\bar{F}\!\! & \!\!=\!\! & \!\!2e^{-f}f_{4}\,\bar{e}%
^{123}-2e^{-f}f_{3}\,\bar{e}^{124}-2te^{-2f}\bar{e}^{126}+2e^{-f}f_{2}\,\bar{%
e}^{134}+\frac{t}{s}e^{-2f}\bar{e}^{135} \\[7pt]
\!\! & \!\!-\!\! & \!\! \frac{t}{s}e^{-2f}\bar{e}^{146}-2e^{-f}f_{1}\,\bar{e}%
^{234}+\frac{t}{s}e^{-2f}\bar{e}^{236}+\frac{t}{s}e^{-2f}\bar{e}%
^{245}+2te^{-2f}\bar{e}^{346}.%
\end{array}%
\end{equation}%
At this point we define the constant
\begin{equation}  \label{kappa}
\kappa^2=1/2\left (2+1/s^2 \right).
\end{equation}
Letting $f_{ij}=\frac{\partial ^{2}f}{\partial x_{j}\partial x_{i}}$, $%
i,j=1,2,3,4$, a short calculation gives
\begin{equation}  \label{fam1h5-abel-diff-new-torsion}
d\bar{T} =-e^{-4f}\left[ \triangle e^{2f}+4t^{2}\Big(2+\frac{1}{s^{2}}\Big)%
\right] \,\bar{e}^{1234}=-\left[ \triangle e^{2f}+8t^{2}\kappa^2\right] \,{e}%
^{1234},
\end{equation}
where $\triangle e^{2f}=( e^{2f})_{11}+( e^{2f})_{22}+( e^{2f})_{33}+(
e^{2f})_{44}$ is the standard Laplacian on $\mathbb{R}^4$.

\begin{cor}
\label{c:nablamin inst} The $(-)$-connection is an instanton if and only if
the torsion 3-form is closed, $d\bar T=0$, i.e., the dilaton function $f$
satisfies the equality
\begin{equation}  \label{in1}
\triangle e^{2f}+8t^{2}\kappa^2=0.
\end{equation}
\end{cor}

\begin{proof}
Take the trace in \eqref{dtr} and use \eqref{fam1h5-abel-diff-new-torsion}
together with the fact that the holonomy of $\nabla^+$ is contained in $%
SU(3)$ to conclude that $R^-$ satisfies the instanton condition \eqref{in2}
if and only if \eqref{in1} holds.
\end{proof}

\subsection{The first Pontrjagin form of the $(-)$-connection}

The $(-)$-connection of the Hermitian structure $(\bar g,J)$ is defined by the
formula $\nabla ^-=\nabla ^{\bar{g}}-\frac12\bar{T}$, where $\nabla ^{\bar{g}%
}$ is the Levi-Civita connection of the metric $\bar g$ and the
torsion is determined in \eqref{tor5}.

Using the metric $\bar{g}$, let $\{\bar{e}_{1},\ldots ,\bar{e}_{6}\}$ be the
dual to $\{\bar{e}^{1},\ldots ,\bar{e}^{6}\}$ orthonormal basis. From
Koszul's formula, we have that the Levi-Civita connection 1-forms $(\omega ^{%
\bar{g}})_{\bar j}^{\bar i}$ are given by
\begin{multline}
(\omega ^{\bar{g}})_{\bar j}^{\bar i}(\bar{e}_{k})=-\frac{1}{2}\Big(\bar{g}(\bar{e}%
_{i},[\bar{e}_{j},\bar{e}_{k}])-\bar{g}(\bar{e}_{k},[\bar{e}_{i},\bar{e}%
_{j}])+\bar{g}(\bar{e}_{j},[\bar{e}_{k},\bar{e}_{i}])\Big)  \label{lc} \\
=\frac{1}{2}\Big(d\bar{e}^{i}(\bar{e}_{j},\bar{e}_{k})-d\bar{e}^{k}(\bar{e}%
_{i},\bar{e}_{j})+d\bar{e}^{j}(\bar{e}_{k},\bar{e}_{i})\Big)
\end{multline}%
taking into account $\bar{g}(\bar{e}_{i},[\bar{e}_{j},\bar{e}_{k}])=-d\bar{e}%
^{i}(\bar{e}_{j},\bar{e}_{k})$. With the help of \eqref{lc} we
compute the expressions for the connection 1-forms $(\omega
^{-})_{\bar j}^{\bar i}$ of the connection $\nabla ^{-}$,
\begin{equation}
(\omega ^{-})_{\bar j}^{\bar i}=(\omega ^{\bar{g}})_{\bar j}^{\bar i}-\frac{1}{2}(\bar{T}%
)_{\bar j}^{\bar i},\qquad\text{ where }\qquad (\bar{T})_{\bar j}^{\bar i}(\bar{e}_{k})=\bar{T}(%
\bar{e}_{i},\bar{e}_{j},\bar{e}_{k}).  \label{minus}
\end{equation}%
Now, \eqref{minus}, \eqref{lc} and \eqref{tor5} show that the non-zero
connection 1-forms $(\omega ^{-})_{\bar j}^{\bar i}$ are given in terms of the basis $%
\{{\bar e}^{1},\ldots ,{\bar e}^{6}\}$ by

{\small
\begin{equation}\label{connection-1-form-for-mu}
\begin{array}{l}
(\omega ^-{\ )}_{\bar 2}^{\bar 1}=e^{-f}\left[ f_{2}\,\bar{e}^{1}-f_{1}\,\bar{e}%
^{2}+f_{4} \,\bar{e}^{3}-f_{3}\,\bar{e}^{4}\right] , \quad 
(\omega ^-{\ )}_{\bar 3}^{\bar 1}=e^{-f}\left[ f_{3}\,\bar{e}^{1}-f_{4}\,%
\bar{e}^{2}-f_{1}\,\bar{e}^{3}+f_{2} \,\bar{e}^{4} \right], \\[2pt]
(\omega ^-{\ )}_{\bar 4}^{\bar 1}=e^{-f}\left[ f_{4}\,\bar{e}^{1}+f_{3} \,%
\bar{e}^{2}-f_{2} \,\bar{e}^{3}-f_{1}\,\bar{e}^{4}\right], \quad
(\omega ^-{\ )}_{\bar 5}^{\bar 1}=-e^{-2f}\frac{t}{s}\,\bar{e}%
^{3},
\\[7pt] (\omega ^-{\ )}_{\bar 6}^{\bar 1}=e^{-2f}2t\left[ \,\bar{e}^{2}+%
\frac{1}{2s}\bar{e}^{4}\right] \,, \quad

(\omega ^-{\ )}_{\bar 3}^{\bar 2}=e^{-f}\left[ f_{4} \,\bar{e}^{1}+f_{3}\,%
\bar{e}^{2}-f_{2}\,\bar{e}^{3}-f_{1} \,\bar{e}^{4}\right] , \\[7pt]

(\omega ^-{\ )}_{\bar 4}^{\bar 2}=e^{-f}\left[ -f_{3} \,\bar{e}^{1}+f_{4}\,%
\bar{e}^{2}+f_{1} \,\bar{e}^{3}-f_{2}\,\bar{e}^{4}\right] , \quad

(\omega ^-{\ )}_{\bar 5}^{\bar 2}=-e^{-2f}\frac{t}{s}\,\bar{e}%
^{4},
\\[7pt] (\omega ^-{\ )}_{\bar 6}^{\bar 2}=e^{-2f}2t\left[ -\,\bar{e}^{1}-%
\frac{1}{2s}\,\bar{e}^{3}\right] , \quad

(\omega ^-{\ )}_{\bar 4}^{\bar 3}=e^{-f}\left[ f_{2}
\,\bar{e}^{1}-f_{1}
\,\bar{e}^{2}+f_{4}\,\bar{e}^{3}-f_{3}\,\bar{e}^{4}\right] , \\[7pt]

(\omega ^-{\ )}_{\bar 5}^{\bar 3}=e^{-2f}\frac{t}{s}\,\bar{e}%
^{1},\qquad (\omega ^-{\ )}_{\bar 6}^{\bar 3}=e^{-2f}2t\left[ \frac{1}{2s}\,\bar{%
e}^{2}-\,\bar{e}^{4}\right] , \\[7pt]

(\omega ^-{\ )}_{\bar 5}^{\bar 4}=e^{-2f}\frac{t}{s}\,\bar{e}%
^{2},\qquad (\omega ^-{\ )}_{\bar 6}^{\bar 4}=e^{-2f}2t\left[ -\frac{1}{2s}\,%
\bar{e}^{1}+\,\bar{e}^{3}\right] .%
\end{array}%
\end{equation}
}

A long straightforward calculation
using \eqref{connection-1-form-for-mu} gives in terms of the basis $\{{\bar e%
}^{1},\ldots ,{\bar e}^{6}\}$ the following formulas for the
curvature 2-forms of $\nabla ^{-}$

{\small
\begin{equation*}
\begin{array}[t]{ll}
(\Omega ^{-})_{\bar 2}^{\bar 1}\!\!= & \!%
\!-(f_{11}+f_{22}+2f_{3}^{2}+2f_{4}^{2}+4t^{2}e^{-2f})\,e^{-2f}\overline{e}%
^{12}+(f_{14}-f_{23}+2f_{2}f_{3}-2f_{1}f_{4})e^{-2f}(\overline{e}^{13}+%
\overline{e}^{24}) \\
\!\! & -(f_{13}+f_{24}-2f_{1}f_{3}-2f_{2}f_{4}+2t^{2}/se^{-2f})e^{-2f}(%
\overline{e}^{14}-\overline{e}^{23}) \\
& -\,(2f_{1}^{2}+2f_{2}^{2}+f_{33}+f_{44}+2t^{2}/s^{2}e^{-2f})e^{-2f}%
\overline{e}^{34}, \\
(\Omega ^{-})_{\bar 3}^{\bar 1}\!\!= & \!%
\!-(f_{14}+f_{23}-2f_{2}f_{3}-2f_{1}f_{4})e^{-2f}(\overline{e}^{12}-%
\overline{e}%
^{34})-(f_{11}+2f_{2}^{2}+f_{33}+2f_{4}^{2}+t^{2}/s^{2}e^{-2f})e^{-2f}%
\overline{e}^{13} \\
\!\! & \!\!+(f_{12}-2f_{1}f_{2}-f_{34}+2f_{3}f_{4})e^{-2f}(\overline{e}^{14}-%
\overline{e}^{23})\!\! \\
&
+(2f_{1}^{2}+f_{22}+2f_{3}^{2}+f_{44}+t^{2}/s^{2}e^{-2f}+4t^{2}e^{-2f})e^{-2f}%
\overline{e}^{24}, \\
(\Omega ^{-})_{\bar 4}^{\bar 1}\!\!= & \!%
\!(f_{13}-f_{24}-2f_{1}f_{3}+2f_{2}f_{4}-2t^{2}/se^{-2f})e^{-2f}(\overline{e}%
^{12}-\overline{e}^{34})-(f_{12}-2f_{1}f_{2}+f_{34}-2f_{3}f_{4})e^{-2f}(%
\overline{e}^{13}+\overline{e}^{24}) \\
\!\! & \!\!-(f_{11}+2f_{2}^{2}+2f_{3}^{2}+f_{44}+t^{2}/s^{2}e^{-2f})\,e^{-2f}%
\overline{e}%
^{14}-(2f_{1}^{2}+f_{22}+f_{33}+2f_{4}^{2}+t^{2}/s^{2}e^{-2f}+4t^{2}e^{-2f})e^{-2f}%
\overline{e}^{23}, \\
(\Omega ^{-})_{\bar 5}^{\bar 1}\!\!= & \!\!2e^{-3f}t/s\left[ f_{4}(\overline{e}^{12}-%
\overline{e}^{34})+f_{1}(e^{13}+e^{24})-f_{2}(\overline{e}^{14}-\overline{e}%
^{23})\right] , \\
(\Omega ^{-})_{\bar 6}^{\bar 1}\!\!= & \!\!2e^{-3f}t/s\left[ \!\!(f_{3}-2f_{1}s)(%
\overline{e}^{12}-\overline{e}^{34})-(f_{2}-2f_{4}s)(\overline{e}^{13}+%
\overline{e}^{24})-(f_{1}+2f_{3}s)(\overline{e}^{14}-\overline{e}^{23})%
\right] , \\
(\Omega ^{-})_{\bar 3}^{\bar 2}\!\!= & \!%
\!(f_{13}-f_{24}-2f_{1}f_{3}+2f_{2}f_{4}+2t^{2}/se^{-2f})e^{-2f}(\overline{e}%
^{12}-\overline{e}^{34})-(f_{12}-2f_{1}f_{2}+f_{34}-2f_{3}f_{4})e^{-2f}(%
\overline{e}^{13}+\overline{e}^{24}) \\
\!\! & \!\!%
\!-(f_{11}+2f_{2}^{2}+2f_{3}^{2}+f_{44}+t^{2}/s^{2}e^{-2f}+4t^{2}e^{-2f})e^{-2f}%
\overline{e}^{14}%
\!-(2f_{1}^{2}+f_{22}+f_{33}+2f_{4}^{2}+t^{2}/s^{2}e^{-2f})e^{-2f}\overline{e%
}^{23}, \\
(\Omega ^{-})_{\bar 4}^{\bar 2}\!\!= & \!%
\!(f_{14}+f_{23}-2f_{2}f_{3}-2f_{1}f_{4})e^{-2f}(\overline{e}^{12}-\overline{%
e}%
^{34})+(f_{11}+2f_{2}^{2}+f_{33}+2f_{4}^{2}+t^{2}/s^{2}e^{-2f}+4t^{2}e^{-2f})e^{-2f}%
\overline{e}^{13} \\
\!\! & \!\!-(f_{12}-2f_{1}f_{2}-f_{34}+2f_{3}f_{4})e^{-2f}(\overline{e}^{14}-%
\overline{e}%
^{23})-(2f_{1}^{2}+f_{22}+2f_{3}^{2}+f_{44}+t^{2}/s^{2}e^{-2f})e^{-2f}\,%
\overline{e}^{24}, \\
(\Omega ^{-})_{\bar 5}^{\bar 2}\!\!= & \!\!2e^{-3f}t/s\left[ -f_{3}(\overline{e}^{12}-%
\overline{e}^{34})+f_{2}(\overline{e}^{13}+\overline{e}^{24})+f_{1}(%
\overline{e}^{14}-\overline{e}^{23})\right] , \\
(\Omega ^{-})_{\bar 6}^{\bar 2}\!\!= & \!\!2e^{-3f}t/s\left[ (f_{4}-2f_{2}s)(\overline{%
e}^{12}-\overline{e}^{34})+(f_{1}-2f_{3}s)(\overline{e}^{13}+\overline{e}%
^{24})-(f_{2}+2f_{4}s)(\overline{e}^{14}-\overline{e}^{23})\right] , \\
(\Omega ^{-})_{\bar 4}^{\bar 3}\!\!= & \!\!-%
\,(f_{11}+f_{22}+2f_{3}^{2}+2f_{4}^{2}+2t^{2}/s^{2}e^{-2f})e^{-2f}\overline{e%
}^{12}+(f_{14}-f_{23}+2f_{2}f_{3}-2f_{1}f_{4})e^{-2f}(\overline{e}^{13}+%
\overline{e}^{24}) \\
\!\! & \!\!-(f_{13}+f_{24}-2f_{1}f_{3}-2f_{2}f_{4}-2t^{2}/se^{-2f})e^{-2f}(%
\overline{e}^{14}-\overline{e}%
^{23})-(2f_{1}^{2}+2f_{2}^{2}+f_{33}+f_{44}+4t^{2}e^{-2f})\,e^{-2f}\overline{%
e}^{34}, \\
(\Omega ^{-})_{\bar 5}^{\bar 3}\!\!= & \!\!2e^{-3f}t/s\left[ f_{2}(\overline{e}^{12}-%
\overline{e}^{34})+f_{3}(\overline{e}^{13}+\overline{e}^{24})+f_{4}(%
\overline{e}^{14}-\overline{e}^{23})\right] , \\
(\Omega ^{-})_{\bar 6}^{\bar 3}\!\!= & \!\!2e^{-3f}t/s\left[ -(f_{1}+2f_{3}s)(%
\overline{e}^{12}-\overline{e}^{34})+(f_{4}+2f_{2}s)(\overline{e}^{13}+%
\overline{e}^{24})-(f_{3}-2f_{1}s)(\overline{e}^{14}-\overline{e}^{23})%
\right] , \\
(\Omega ^{-})_{\bar 5}^{\bar 4}\!\!= & 2e^{-3f}t/s\left[ \!\!-f_{1}(\overline{e}^{12}-%
\overline{e}^{34})+f_{4}(\overline{e}^{13}+\overline{e}^{24})-f_{3}(%
\overline{e}^{14}-\overline{e}^{23})\right] , \\
(\Omega ^{-})_{\bar 6}^{\bar 4}\!\!= & \!2e^{-3f}t/s\left[ \!-(f_{2}+2f_{4}s)(%
\overline{e}^{12}-\overline{e}^{34})-(f_{3}+2f_{1}s)(\overline{e}^{13}+%
\overline{e}^{24})-(f_{4}-2f_{2}s)(\overline{e}^{14}-\overline{e}^{23})%
\right] , \\
(\Omega ^{-})_{\bar 6}^{\bar 5}\!\!= & 2e^{-4f}t^{2}/s^{2}\!\!\left[ -(\overline{e}%
^{12}-\overline{e}^{34})+2s(\overline{e}^{14}-\overline{e}^{23})\right].%
\end{array}%
\end{equation*}%
}

\begin{prop}
The first Pontrjagin form of $\nabla ^{-}$ is a scalar multiple of $e^{1234}$
given by {\small
\begin{equation}  \label{fam1h5-p1}
\pi ^{2}p_{1}(\nabla ^{-}) =\left[ \sum_{1\leq i<j\leq 4}\left(
det(f_{ij})+(f_{i}^{2}f_{j})_{j}+(f_{i}f_{j}^{2})_{i}\right)
+\sum_{i=1}^{4}(f_{i}^{3})_{i}-\frac{3}{2}t^{2}\kappa ^{2}\triangle e^{-2f}%
\right] {e}^{1234}.
\end{equation}%
}
\end{prop}

\begin{proof}
The proof of \eqref{fam1h5-p1} is a long straightforward
calculations using the formulas for the curvature 2-form of
$\nabla^-$.
\end{proof}
{Note that even though the curvature 2-forms of $\nabla^-$ are quadratic in
the gradient of the dilaton, remarkably, the Pontrjagin form of $\nabla^-$
is also quadratic in these terms. }

\subsection{A conformally compact solution with negative $\protect\alpha^{\prime}$. Proof of Theorem~\protect\ref{t:cpmct strominger}}
 \label{ss:compact ex}

Here we give the proof of Theorem \ref{t:cpmct strominger}.

By \cite{y3} there is no compact solution of Strominger's system for
positive $\alpha^{\prime }$ in the case of the Chern connection on torus
bundle over $\mathbb{T}^4$. On the other hand, the existence of a solution on a
2-torus bundle over K3-surfaces given in \cite{y3} seems to depend on the
assumption $\alpha^{\prime }>0$, 
{\ whereas the existence of a solution with negative $\alpha^{\prime
}$ is not clear.
} 

The proof of Theorem \ref{t:cpmct strominger} occupies the remaining part
of Section \ref{ss:compact ex}. We begin with a Proposition defining the
instanton bundle.

\begin{prop} \label{instanton2} Let $A_{\lambda }$, $\lambda
=(\lambda _{1},\lambda _{2},\lambda _{3})\in \mathbb{R}^{3}$ be
the linear connection on $H_5$ whose non-zero 1-forms are given as
follows
\begin{equation*}
\begin{array}{l}
(\omega ^{A_{\lambda }})_{\bar 2}^{\bar 1}=-(\omega ^{A_{\lambda
}})_{\bar 1}^{\bar 2}=-(\omega ^{A_{\lambda }})_{\bar 4}^{\bar 3}=(\omega ^{A%
_{\lambda }})_{\bar 3}^{\bar 4}=-\lambda _{1}\,{\bar e}^{6}, \\[8pt]
(\omega ^{A_{\lambda }})_{\bar 3}^{\bar 1}=-(\omega ^{A_{\lambda
}})_{\bar 1}^{\bar 3}=(\omega ^{A_{\lambda }})_{\bar 4}^{\bar 2}=-(\omega ^{A%
_{\lambda }})_{\bar 2}^{\bar 4}=-\lambda _{2}\,{\bar e}^{6}, \\[8pt]
(\omega ^{A_{\lambda }})_{\bar 4}^{\bar 1}=-(\omega ^{A_{\lambda
}})_{\bar 1}^{\bar 4}=-(\omega ^{A_{\lambda }})_{\bar 3}^{\bar 2}=(\omega ^{A%
_{\lambda }})_{\bar 2}^{\bar 3}=-\lambda _{3}\,{\bar e}^{6}.%
\end{array}%
\end{equation*}%
Then, $A_{\lambda }$ is an $SU(3)$-instanton which preserves the
metric. Furthermore, the first Pontrjagin form of $A_{\lambda }$
is
\begin{equation}
8\pi ^{2}p_{1}(A_{\lambda })=-8t^{2}(1+\kappa ^{2})|\lambda |^{2}\,{e}%
^{1234}, \qquad |\lambda|^2=\lambda_1^2+\lambda_2^2+\lambda_3^2.
\label{abinst}
\end{equation}
\end{prop}

\begin{proof}
{\ A direct calculation shows that the non-zero curvature forms $(\Omega ^{%
A_{\lambda }})_{\bar j}^{\bar i}$ of the connection $A_{\lambda }$
are:
\begin{equation*}
\begin{array}{l}
(\Omega ^{A_{\lambda }})_{\bar 2}^{\bar 1}=-(\Omega ^{A_{\lambda
}})_{\bar 1}^{\bar 2}=-(\Omega ^{A_{\lambda }})_{\bar 4}^{\bar 3}=(\Omega ^{A%
_{\lambda }})_{\bar 3}^{\bar 4}=-\lambda _{1}\,d\bar e^{6}=2t\lambda _{1}e^{-2f}(\bar e^{12}-\bar e^{34})+%
\frac{t}{s}\lambda _{1}e^{-2f}(\bar e^{14}-\bar e^{23}), \\[8pt]
(\Omega ^{A_{\lambda }})_{\bar 3}^{\bar 1}=-(\Omega ^{A_{\lambda
}})_{\bar 1}^{\bar 3}=(\Omega ^{A_{\lambda }})_{\bar 4}^{\bar 2}=-(\Omega ^{A%
_{\lambda }})_{\bar 2}^{\bar 4}=-\lambda _{2}\,d\bar e^{6}=2t\lambda _{2}e^{-2f}(\bar e^{12}-\bar e^{34})+%
\frac{t}{s}\lambda _{2}e^{-2f}(\bar e^{14}-\bar e^{23}), \\[8pt]
(\Omega ^{A_{\lambda }})_{\bar 4}^{\bar 1}=-(\Omega ^{A_{\lambda
}})_{\bar 1}^{\bar 4}=-(\Omega ^{A_{\lambda }})_{\bar 3}^{\bar 2}=(\Omega ^{A%
_{\lambda }})_{\bar 2}^{\bar 3}=-\lambda _{3}\,d\bar e^{6}=2t\lambda _{3}e^{-2f}(\bar e^{12}-\bar e^{34})+%
\frac{t}{s}\lambda _{3}e^{-2f}(\bar e^{14}-\bar e^{23}).%
\end{array}%
\end{equation*}%
It is straightforward to see that $A_{\lambda }$ satisfies
\eqref{in2} and therefore it is an $SU(3)$-instanton. After
another lengthy calculation we see $8\pi ^{2}p_{1}(A_{\lambda
})=-4t^{2}(4+1/s^{2})|\lambda
|^{2}e^{-4f}\,\bar{e}^{1234}$, which in view of \eqref{kappa} and %
\eqref{conf} implies formula \eqref{abinst}. }
\end{proof}

Now, we suppose that the function $f$ depends on one variable, say $%
f=f(x^{1})$. Using $\left( f^{\prime \prime }-2f^{\prime 2}\right) e^{-2f}=-%
\frac{1}{2}\left( e^{-2f}\right) ^{\prime \prime }$ we have from %
\eqref{fam1h5-p1}
\begin{equation}
8\pi ^{2}p_{1}(\nabla ^{-})=4\left( 2f^{\prime 3}-3t^{2}\kappa ^{2}\left(
e^{-2f}\right) ^{\prime }\right) ^{\prime }\,{e}^{1234}.  \label{p01}
\end{equation}%
Furthermore, from \eqref{fam1h5-abel-diff-new-torsion}
\begin{equation}
d\bar{T}=-\left( \left( e^{2f}\right) ^{\prime \prime }+8t^{2}\kappa
^{2}\right) \,{e}^{1234}.  \label{tor}
\end{equation}%
In view of \eqref{abinst}, \eqref{p01} and \eqref{tor} the anomaly
cancellation condition \eqref{acgen}, i.e., $d\bar{T}=\frac{\alpha ^{\prime }%
}{4}8\pi ^{2}\Big(p_{1}(\nabla ^{-})-p_{1}(A_{\lambda })\Big)$,
takes the form of a single ODE for the function $f$
\begin{equation}
\left( \left( e^{2f}\right) ^{\prime }-3\alpha ^{\prime }t^{2}\kappa
^{2}\left( e^{-2f}\right) ^{\prime }+2\alpha ^{\prime }f^{\prime 3}\right)
^{\prime }+8t^{2}\kappa ^{2}+2\alpha ^{\prime }t^{2}(1+\kappa ^{2})|\lambda
|^{2}=0.  \label{eqcan}
\end{equation}%
For a negative $\alpha ^{\prime }$ we choose $\kappa ^{2}$ or $|\lambda |^{2}
$ so that $8t^{2}\kappa ^{2}+2\alpha ^{\prime }t^{2}(1+\kappa ^{2})|\lambda
|^{2}=0$, i.e., we let
\begin{equation*}
\alpha ^{\prime }=-\alpha^2,\qquad 4\kappa ^{2}=\alpha ^{2}(1+\kappa
^{2})|\lambda |^{2},
\end{equation*}
which simplifies \eqref{eqcan} to the ordinary differential equation

\begin{equation}  \label{solv4}
\left( e^{2f}\right) ^{\prime }+3\alpha ^{2}t^{2}\kappa ^{2}\left(
e^{-2f}\right) ^{\prime }-2\alpha ^{2 }f^{\prime 3}=A=const.
\end{equation}%
At this point we let $u=\alpha^{-2} e^{2f}$. With this substitution the
left-hand side of \eqref{solv4} becomes 
\begin{equation*}
\left( e^{2f}\right) ^{\prime }-3\alpha ^{\prime }t^{2}\left( e^{-2f}\right)
^{\prime }+2\alpha ^{\prime }f^{\prime 3}=\frac{\alpha^2 u^{\prime }}{4u^{3}}%
\left( 4u^{3}-12\frac {t^{2}\kappa ^{2}}{\alpha ^{2 }}u-u^{\prime
2}\right) .
\end{equation*}%
For $A=0$ consider the following ordinary differential equation for the
function $u=u(x^1)>0$ 
\begin{equation}  \label{solv5}
u^{\prime 2}={4}u^{3}-12\frac {t^{2}\kappa ^{2}}{\alpha ^{2 }}u=4u\left( u-a\right) \left(
u+a\right) ,\qquad a=\kappa |t|\sqrt{3}/ \alpha.
\end{equation}%
Equation \eqref{solv5} can also be considered in the complex plane by
replacing the real derivative with the complex derivative which turns it
into the Weierstrass' equation
\begin{equation*}  
\left (\frac {d\, \mathcal{P}}{dz}\right )^2=4\mathcal{P}\left( \mathcal{P}%
-a\right) \left( \mathcal{P}+a\right)
\end{equation*}
for the doubly periodic Weierstrass $\mathcal{P}$ function with a pole at
the origin where it has the expansion
\begin{equation*}
\mathcal{P}(z) = \frac {1}{z^2} + \frac {a^2}{5} z^2 + bz^6 +
\cdots,
\end{equation*}
(no $z^4$ term and only even powers). In addition, as well known
\cite{Erd} and \cite{Ahl}, letting $\tau_{\pm }$ be the basic
half-periods such that
$tau _{+}$ is real and $\tau _{-}$ is purely imaginary we have that $%
\mathcal{P}$ is real valued on the lines $\mathfrak{R}\mathfrak{e}%
\,z=m\tau _{+}$ or $\mathfrak{I}\mathfrak{m}\,z=im\tau _{-}$,
$m\in \mathbb{Z}$. Furthermore, in the fundamental region centered
at the origin, where $\mathcal{P}$  has a pole of order two, we
have that $\mathcal{P}(z)$ decreases from $+\infty $ to $a$ to $0$
to $-a$ to $-\infty $ as $z$ varies
along the sides of the half-period rectangle from $0$ to $\tau _{+}$ to $%
\tau _{+}+\tau _{-}$ to $\tau _{-}$ to $0$.

Thus, $u(x^1)=\mathcal{P}(x^1)$ defines a non-negative
$2\tau_{+}$-periodic function with singularities at the points $2n \tau_{+}$,
$n\in \mathbb{Z}$, which solves the real equation \eqref{solv5}.
From the Laurent expansion of the Weierstrass' function it follows
\begin{equation*}
u(x_1)=\frac {1}{(x^1)^2}\left (1+ \frac {a^2}{5} (x^1)^4 + \cdots
\right).
\end{equation*}
By construction, $f= \frac 12 \ln (\alpha^2 u)$ is a periodic function with
singularities on the real line which is a solution to equation \eqref{eqcan}
sufficient for the anomaly cancellation condition. {\ 
Therefore the $SU(3)$ structure defined by $\bar F$ and the
non-degenerate (3,0) form $\bar {\Psi} = (\bar e^1+i\bar
e^2)\wedge(\bar e^3+i \bar e^4)\wedge (\bar e^5 +i \bar e^6)$
descends to the $6$-dimensional nilmanifold $M^{6}=\Gamma
\backslash H_{5}$ with singularity, determined by the singularity of $u$, where $H_{5}$
is the 2-step nilpotent Lie group with Lie algebra $\mathfrak{h}_{5}$, defined by %
\eqref{fam1h5}, and $\Gamma $ is a lattice with the same period as $f$, i.e., $2 \tau_{+}$ in all variables. In
fact, as seen from the asymptotic behavior of $u$, $M^6$ is the total space
of a $\mathbb{T}^2$ bundle over the asymptotically hyperbolic manifold $M^4$ with
metric
\begin{equation*}
\bar g_H =u(x^1)\left ( (dx^1)^2+(dx^2)^2 + (dx^3)^2 +(dx^4)^2
\right ),
\end{equation*}
which is a conformally compact 4-torus with conformal boundary at infinity a
flat 3-torus. Thus, we 
conclude that there is a complete solution with non-constant dilaton,
non-trivial instanton and flux and with a negative $\alpha ^{\prime }$
parameter. This completes the proof of Theorem \ref{t:cpmct strominger}. }

A few remarks are in order. First, since the function $u$ has a
$\mathbb{Z}_2 $-symmetry determined by the symmetry with respect
to the line $x^1=\tau_+$ we also obtain a solution on the quotient
$M^6/\mathbb{Z}_2$.

Second, the function $v(x_1)=\mathcal{P}(\tau _{+} +ix_1)$, which
is the
restriction of $\mathcal{P}$ to the line $\mathfrak{R}\mathfrak{e}%
\,z=\tau _{+}$ leads to a solution ``equivalent'' to the one
described above, taking into account the invariance under
translation in $x_1$. Indeed, clearly $(v^{\prime })^2=-\left
(\frac {d\, \mathcal{P}}{dz}\right )^2$ hence $v$ satisfies
\begin{equation*} 
(v^{\prime })^2=-4v\left( v-a\right) \left( v+a\right) ,\qquad a=\kappa |t|%
\sqrt{3}/ \alpha.
\end{equation*}
The above equation is \eqref{solv5} with the substitution $u=\frac
{a^2}{v}$ which shows that we obtain the same six dimensional
$SU(3)$ structure. We note that $v$ is an even non-negative
periodic function with period $2i\tau _{-}$ (without a loss of
generality $i\tau _{-} >0$) such that $v(-i\tau _{-})=v(i\tau _{-}
)=0$, $v(0)=a$ and $v$ increases on the interval $(-i\tau _{-}
,0)$.



\subsection{Complete solution with positive $\protect\alpha^{\prime}$. Proof of Theorem~\protect\ref{t:main2}}
\label{ss:non-compact ex}

{Let us consider a connection $%
A_{a,d}$ depending on parameters $a,d\in \mathbb{R}$, $d\not=0$, whose
non-zero connection 1-forms $(\omega ^{A_{a,d}})_{\bar j}^{\bar i}$ in the basis $\{%
\bar{e}^{1},\ldots ,\bar{e}^{6}\}$ are as follows {\small
\begin{equation*}  \label{Aad}
\begin{aligned} & (\omega ^{A_{a,d}})_{\bar 2}^{\bar 1} =e^{-f}\left[
f_{2}\,\bar{e}^{1}-f_{1}\,\bar{e}^{2}+f_{4}\,\bar{e}^{3}-f_{3}
\,\bar{e}^{4}\right],\qquad (\omega ^{A_{a,d}})_{\bar 3}^{\bar 1}=
e^{-f}\left[
f_{3}\,\bar{e}^{1}-f_{4}\,\bar{e}^{2}-f_{1}\,\bar{e}^{3}+f_{2}\,\bar{e}^{4}%
\right], \\[2pt] & (\omega ^{A_{a,d}})_{\bar 4}^{\bar 1}=e^{-f}\left[
f_{4}\,\bar{e}^{1}+f_{3}\,\bar{e}^{2}-f_{2}\,\bar{e}^{3}-f_{1}\,\bar{e}^{4}%
\right] ,\qquad (\omega ^{A_{a,d}})_{\bar 5}^{\bar 1}=
-\frac{a}{d}\,e^{-2f}\bar{e}^{3}, \\[2pt] & (\omega ^{A_{a,d}})_{\bar 6}^{\bar 1} =
ae^{-2f}\left[ 2\,\,\bar{e}^{2}+1/d\, \bar{e}^{4}\right] ,\qquad
(\omega ^{A_{a,d}})_{\bar 3}^{\bar 2}= e^{-f}\left[
f_{4}\,\bar{e}^{1}+f_{3}\,\bar{e}^{2}-f_{2}\,\bar{e}^{3}-f_{1}\,\bar{e}^{4}%
\right] , \\[2pt] & (\omega ^{A_{a,d}})_{\bar 4}^{\bar 2}=e^{-f} \left[
-f_{3}\,\bar{e}^{1}+f_{4}\,\bar{e}^{2}+f_{1}\,\bar{e}^{3}-f_{2}\,\bar{e}^{4}%
\right] ,\qquad (\omega ^{A_{a,d}})_{\bar 5}^{\bar
2}=-\frac{a}{d}\, e^{-2f}\bar{e}^{4}, \\ 
& (\omega
^{A_{a,d}})_{\bar 6}^{\bar 2}=-a\,e^{-2f}\left[
2\,\bar{e}^{1}+\frac{1}{d}\bar{e}^{3}\right] ,\qquad
(\omega ^{A_{a,d}})_{\bar 4}^{\bar 3}=e^{-f}\left[
f_{2}\,\bar{e}^{1}-f_{1}\,\bar{e}^{2}+f_{4}\,\bar{e}^{3}-f_{3}\,\bar{e}^{4}%
\right], \\ 
& (\omega ^{A_{a,d}})_{\bar 5}^{\bar 3}= \frac{a}{d}
\, e^{-2f}\bar{e}^{1},\qquad (\omega ^{A_{a,d}})_{\bar 6}^{\bar
3}=ae^{-2f}\left[ \frac{1}{d}\,\bar{e}^{2}-2\,\bar{e}^{4}\right],
\\ & (\omega ^{A_{a,d}})_{\bar 5}^{\bar 4}=
\frac{a}{d}\,e^{-2f}\,\bar{e}^{2},\qquad (\omega ^{A_{a,d}})_{\bar
6}^{\bar 4}=-ae^{-2f}\left
[\frac{1}{d}\,\bar{e}^{1}-2\,\bar{e}^{3}\right].\end{aligned}
\end{equation*}
}}
\begin{lemma}
\label{l:insab} $A_{a,d}$ is an instanton with respect to the
$SU(3)$ structure defined with the help of the basis \eqref{conf} if
and only if the dilaton function satisfies
\begin{equation}  \label{insab}
\triangle e^{2f}=-8\tau ^{2}a^{2},\qquad \tau ^{2}=\frac{1}{2}\left( 2+\frac{%
1}{d^{2}}\right).
\end{equation}
\end{lemma}

\begin{proof}
Observe that the connection 1-forms of $A_{a,d}$ are given by %
\eqref{connection-1-form-for-mu} replacing $t$ with $a$ and $s$ with $d$.
Then the assertion follows from Corollary~\ref{c:nablamin inst}.
\end{proof}

Lemma \ref{l:insab} shows, in particular, that $A_{a,d}$ is an instanton
with respect to the SU(3) structure determined by the basis \eqref{conf} if
\begin{equation}
e^{2f}=h(x)-a^{2}\tau ^{2}|x|^{2},  \label{insa-sol}
\end{equation}%
where $h$ is a harmonic function on $\mathbb{R}^4$.

{\ The expression \eqref{fam1h5-p1} yield that the difference between the
first Pontrjagin forms of $\nabla^-$ and $A_{a,d}$ is given by the formula
\begin{multline}  \label{p11}
8\pi^2\Big(p_1(\nabla^-)-p_1(A_{a,d})\Big) =\frac{12}{d^2s^2}\Big[%
a^2(1+2d^2)s^2-d^2(1+2s^2)t^2\Big] e^{-6f}\Big (2\vert f\vert^2 - \triangle f%
\Big)\, \bar{e}^{1234} \\
= -24\left ( t^2\kappa^2 -a^2\tau^2\right)e^{-4f}\left (\triangle
e^{-2f}\right)\, \bar{e}^{1234}=-24\left ( t^2\kappa^2
-a^2\tau^2\right)\left (\triangle e^{-2f}\right)\, {e}^{1234}.
\end{multline}
On the other hand recalling \eqref{fam1h5-abel-diff-new-torsion} and taking
into account \eqref{p11}, we have that the anomaly cancellation condition
\begin{multline*}
d\bar T-\frac{\alpha^{\prime }}48\pi^2\Big(p_1(\nabla^-)-p_1(A_{a,d})\Big) %
=-\left [ \triangle e^{2f}+8t^2\kappa^2 - 3{\alpha^{\prime }}
(t^2\kappa^2-a^2\tau^2) \triangle e^{-2f} \right]\, {e}^{1234}=0
\end{multline*}
simplifies to the single equation for the dilaton
$
\triangle e^{2f}+8t^2\kappa^2 - 3{\alpha^{\prime }}
(t^2\kappa^2-a^2\tau^2) \triangle e^{-2f}=0.$ 

Thus a non-trivial dilaton is given by \eqref{insa-sol} which
satisfies the equation
\begin{equation}  \label{anomaly cancel}
(t^2\kappa^2-a^2\tau^2)\left (8 - 3{\alpha^{\prime }} \triangle
e^{-2f}\right)=0.
\end{equation}
We analyze the solution in the next two cases. }

{\normalsize \textbf{Case 1.} Here $t^2\kappa^2-a^2\tau^2=0$, hence by %
\eqref{p11} the anomaly condition is trivially satisfied for any $%
\alpha^{\prime }$, provided the torsion is closed. In this case the solution
is given by the solutions of \eqref{insab}. Furthermore, taking into account
Corollary \ref{c:nablamin inst} both $\nabla^{-}$ and $A_{a,d}$ are
instantons. For example, a particular case of \eqref{insa-sol} is the
solution
\begin{equation*}
e^{2f}=a^2\tau^2(1-|x|^2)
\end{equation*}
defined in the unit ball. }

Notice that in this case we also obtain a solution of the type II theory,
see \cite{CCDLMZ} for the case of the Iwasawa manifold and \cite[Section VII]%
{GMW} for the general case of two flat directions fibered over a four
dimensional base $M_0$.

\textbf{Case 2.} Here $t^2\kappa^2-a^2\tau^2\not=0$, hence the anomaly
condition is non-trivial.


We need to solve the system of the two equations \eqref{insab} and %
\eqref{anomaly cancel}. To get a solution we take $a=0$ in \eqref{insab} and
arrive to the next two equations for the dilaton $f$: 
\begin{equation*}
\triangle e^{2f}=0, \qquad \triangle e^{-2f}=8/(3\alpha^{\prime }).
\end{equation*}
Hence the solution with a singularity is given by
\begin{equation}  \label{fundsol}
e^{2f}= \frac {3\alpha^{\prime }}{|x-b|^2}, \quad b\in \mathbb{R}^4,
\end{equation}
As a result of the above arguments we obtain a non-compact solution with
non-constant dilaton, non-trivial instanton and flux with positive $%
\alpha^{\prime }$. This solution is similar to the multi-instanton solution
considered in \cite{CHS}. 
Taking into account that $H_5$ is a $\mathbb{R}^2$-bundle over $\mathbb{R}^4$%
, and using logarithmic radial coordinates near the singularity as in \cite%
{CHS} it follows that  
the $4-D$ metric induced on $\mathbb{R}^4$ is actually complete. In fact,
taking the singularity at the origin, in the coordinate $q=\sqrt {%
3\alpha^{\prime }}/2 \, \ln \left( {|x|^2}/{3\alpha^{\prime }}\right)=-\sqrt{%
3\alpha^{\prime }}\, f$, we have
that the dilaton and the $4-D$ metric can be expressed as follows
\begin{equation*}
\bar g_H=\sum_{i=1}^4 e^{2f}(e^i)^2=dq^2+3\alpha^{\prime }ds^2_3,\qquad f=-q
\sqrt {3\alpha^{\prime }},
\end{equation*}
where $ds^2_3$ is the metric on the unit three-dimensional sphere in the
four dimensional Euclidean space. The completeness of the horizontal
metric implies that the metric
\begin{equation*}
\bar g=\bar g_H +(e^5)^2+(e^6)^2
\end{equation*}
is also complete. 
This finishes the proof of Theorem \ref{t:main2}.

\section{Contraction and the Lie algebra $\mathfrak{h}_3$}
{\ \label{s:alg all 6-D} }

Taking into account that the Lie algebra $\mathfrak{%
h}_3$ is a contraction of the Lie algebra $\mathfrak{h}_5 $ we can obtain solutions for the Lie algebra $\mathfrak{h}_3$ from the solution on $\mathfrak{h}_5$. Indeed, letting $%
\frac{t}{s}\rightarrow 0$ in \eqref{fam1h5} we obtain
\begin{equation*}  
de^1 = de^2=de^3=de^4=0,\quad de^5 = 0, \quad de^6 = -2\, t\,(e^{12}-e^{34})
\end{equation*}
which are the structure equations of $\mathfrak{h}_3$.

Correspondingly, using equations \eqref{conf} we obtain the
structure equations
\begin{equation*}  
\begin{aligned} &d \bar{e}^1 = -e^{-f}\left (f_2\, \bar{e}^{12} +f_3\,
\bar{e}^{13} +f_4\, \bar{e}^{14}\right),\quad &d \bar{e}^2 = e^{-f}\left
(f_1\, \bar{e}^{12} - f_3\, \bar{e}^{23} - f_4\, \bar{e}^{24}\right
),\\[7pt] &d \bar{e}^3 = e^{-f}\left (f_1\, \bar{e}^{13} + f_2\,
\bar{e}^{23} - f_4\, \bar{e}^{34}\right ),\quad &d \bar{e}^4 = e^{-f}\left
(f_1\, \bar{e}^{14} + f_2\, \bar{e}^{24} + f_3\, \bar{e}^{34}\right
),\\[7pt] &d \bar{e}^5 = 0,\quad &d \bar{e}^6 = -2t e^{-2f}(\bar{e}^{12} -
\bar{e}^{34}). \end{aligned}
\end{equation*}
As before, we consider the $\nabla^-$ connection described in %
\eqref{connection-1-form-for-mu}. We define the connection $A_a$ by letting
the parameter $d\rightarrow\infty$ in \eqref{Aad}, or equivalently $\frac{a}{%
d}\rightarrow 0$. All remaining calculation in Section
\ref{s:alg5} are valid by taking the above described limits. As a
result, Theorem~\ref{t:cpmct strominger} and
Theorem~\ref{t:main2}
give solutions with non-constant dilaton on $%
\mathfrak{h}_3$.

\begin{rmrk}
\label{r:nablaplusinst} It can be checked from the expression for the curvature $2$-forms of $\nabla^-$ using %
\eqref{dtr} and \eqref{fam1h5-abel-diff-new-torsion} that the connection $%
\nabla^+ $ is an $SU(3)$-instanton if and only if $f$ is a
constant function and $ \frac{t}{s}\rightarrow 0$, i.e. the
connection $\nabla^+ $ is an $SU(3)$-instanton if and only
if $f$ is constant and the group is~$H_3$.
\end{rmrk}

{\bf Acknowledgments.} The work was partially
supported through Project MICINN (Spain) MTM2011-28326-C02-01/02, and Project
of UPV/EHU ref.\ UFI11/52. DV was partially supported by Simons Foundation grant \#279381.


\begin{thebibliography}{99}
\bibitem{Ahl} {\ L.V. Ahlfors, Complex analysis: An introduction of the
theory of analytic functions of one complex variable,
McGraw-Hill Book Co., New York-Toronto-London 1966. }

\bibitem{AG1} {\ B. Andreas, M. Garc\'{i}a-Fern\'andez, \emph{Heterotic
non-K\"ahler geometries via polystable bundles on Calabi-Yau threefols},
J. Geom. Phys. \textbf{62} (2012), 183--188.}

\bibitem{AG2} {\ B. Andreas, M. Garc\'{i}a-Fern\'andez, \emph{Solutions of the
Strominger system via stable bundles on Calabi-Yau threefolds}, Commun. Math.
Phys. \textbf{315} (2012), 153--168. }

\bibitem{AH} {\ M.F. Atiyah, N. Hitchin, \textbf{The Geometry and Dynamics
of Magnetic Monopoles, M. B. Porter Lectures}, Princeton University Press
(1988). }

\bibitem{BBDG} {\ K. Becker, M. Becker, K. Dasgupta, P.S. Green, \emph{%
Compactifications of heterotic theory on non-K\"ahler complex manifolds: I},
JHEP \textbf{0304} (2003), 007. }

\bibitem{BBE} {\ K. Becker, M. Becker, K. Dasgupta, P.S. Green, E. Sharpe,
\emph{Compactifications of heterotic strings on non-K\"ahler complex
manifolds: II}, Nuclear Phys. \textbf{B 678} (2004), 19--100. }

\bibitem{BBDP} {\ K. Becker, M. Becker, K. Dasgupta, S. Prokushkin, \emph{%
Properties of heterotic vacua from superpotentials},
Nuclear Phys. \textbf{B 666} (2003),
144--174. }

\bibitem{y4} {\ K. Becker, M. Becker, J.-X. Fu, L.-S. Tseng, S.-T. Yau, \emph{%
Anomaly cancellation and smooth non-K\"ahler solutions in heterotic string
theory}, Nuclear Phys. \textbf{B 751} (2006), 108--128. }

\bibitem{BBCG} {\ K. Becker, C. Bertinato, Y.-C. Chung, G. Guo, \emph{%
Supersymmetry breaking, heterotic strings and fluxes}, Nuclear Phys. \textbf{%
B 823} (2009), 428--447. }

\bibitem{BSethi} {\ K. Becker, S. Sethi, \emph{Torsional heterotic geometries%
}, Nuclear Phys. \textbf{B 820} (2009), 1--31. }

\bibitem{Berg} {\ E.A. Bergshoeff, M. de Roo, \emph{The quartic effective
action of the heterotic string and supersymmetry}, Nuclear Phys. \textbf{B 328}
(1989), 439--468. }

\bibitem{Bis} {\ I. Biswas, A. Mukherjee, \emph{Solutions of Strominger
system from unitary representations of cocompact lattices of} SL(2,$\mathbb{C}$),
Commun. Math. Phys. \textbf{322} (2013),
373--384. }

\bibitem{CHS} {\ C.G. Callan, J.A. Harvey, A. Strominger,
\emph{Worldsheet approach to heterotic instantons and solitons},
Nuclear Phys. \textbf{B 359} (1991),
611--634. }

\bibitem{CHS1} {\ C.G. Callan, J.A. Harvey, A. Strominger, \emph{%
Worldbrane actions for string solitons}, Nuclear Phys. \textbf{B 367} (1991), 60--82. }

\bibitem{CHSW85} {\ P. Candelas, G.T. Horowitz, A. Strominger, E. Witten,
\emph{Vacuum configurations for superstrings},
Nuclear Phys. \textbf{B 258} (1985),
46--74. }

\bibitem{Car1} {\ G.L. Cardoso, G. Curio, G. Dall'Agata, D. L\"ust, \emph{BPS
action and superpotential for heterotic string compactifications with fluxes}%
, JHEP \textbf{0310} (2003), 004. }

\bibitem{CCDLMZ} {\ G.L. Cardoso, G. Curio, G. Dall'Agata, D. L\"ust, P.
Manousselis, G. Zoupanos, \emph{Non-K\"ahler string backgrounds and their
five torsion classes}, Nuclear Phys. \textbf{B 652} (2003), 5--34. }

\bibitem{CS} {\ S. Chiossi, S. Salamon, \emph{The intrinsic torsion of {\rm SU(3)}
and $G_2$-structures}, Differential Geometry, Valencia 2001, World Sci.
Publishing, 2002, pp. 115--133. }

\bibitem{DRS} {\ K. Dasgupta, G. Rajesh, S. Sethi, \emph{M theory,
orientifolds and $G$-flux},
JHEP \textbf{9908} (1999), 023. }

\bibitem{DFG} {\ K. Dasgupta, H. Firouzjahi, R. Gwyn, \emph{On the warped
heterotic axion}, JHEP \textbf{0806} (2008), 056. }

\bibitem{Bwit} {\ B. de Wit, D.J. Smit, N.D. Hari Dass, \emph{Residual
supersymmetry of compactified D=10 supergravity}, Nuclear Phys. \textbf{B 283}
(1987), 165--191. }

\bibitem{D} S.K. Donaldson, \emph{Anti self-dual Yang-Mills connections
over complex algebraic surfaces and stable vector bundles}, Proc. London
Math. Soc. \textbf{50} (1985),
1--26.

\bibitem{DLu} {\ M.J. Duff, J.X. Lu, \emph{Elementary five-brane solutions
of D = 10 supergravity}, Nuclear Phys. \textbf{B 354} (1991), 141--153. }


\bibitem{Erd} {\ A. Erd\'elyi, W. Magnus, F. Oberhettinger, F.G. Tricomi,
Higher transcendental functions. Vol. I, II and III,
Robert E. Krieger Publishing Co., Inc., Melbourne, Fla., 1981. }


\bibitem{FIUVdim5} {\ M. Fern\'{a}ndez, S. Ivanov, L. Ugarte, R. Villacampa,
\emph{Compact supersymmetric solutions of the heterotic equations of motion
in dimension 5}, Nuclear Phys. \textbf{B 820} (2009), 483--502. }

\bibitem{FIUVdim7-8} {\ M. Fern\'{a}ndez, S. Ivanov, L. Ugarte, R. Villacampa,
\emph{Compact supersymmetric solutions of the heterotic
equations of motion in dimensions 7 and 8}, Adv. Theor. Math. Phys. \textbf{%
15} (2011), 245--284. }

\bibitem{FIUV} {\ M. Fern\'andez, S. Ivanov, L. Ugarte, R. Villacampa, \emph{%
Non-K\"ahler heterotic-string compactifications with non-zero fluxes and
constant dilaton}, Commun. Math. Phys. \textbf{288} (2009), 677--697. }

\bibitem{FTUV} {\ M. Fern\'andez, A. Tomassini, L. Ugarte, R. Villacampa,
\emph{Balanced Hermitian metrics from $SU(2)$-structures},  J. Math. Phys.
\textbf{50} (2009),
033507, 15 pp. }

\bibitem{FLY} {\ J.-X. Fu, L.-S. Tseng, S.-T. Yau, \emph{Local heterotic
torsional models}, Commun. Math. Phys. \textbf{289} (2009), 1151--1169. }

\bibitem{y2} {\ J.-X. Fu, S.-T. Yau, \emph{Existence of supersymmetric
Hermitian metrics with torsion on non-K\"ahler manifolds},
arXiv:hep-th/0509028. }

\bibitem{y3} {\ J.-X. Fu, S.-T. Yau, \emph{The theory of superstring with flux
on non-K\"ahler manifolds and the complex Monge-Amp\`ere equation}, J. Diff.
Geom. \textbf{78} (2008), 369--428. }


\bibitem{GKMW} {\ J.P. Gauntlett, N. Kim, D. Martelli, D. Waldram, \emph{%
Fivebranes wrapped on SLAG three-cycles and related geometry}, JHEP \textbf{0111}
(2001), 018. }

\bibitem{GMPW} {\ J.P. Gauntlett, D. Martelli, S. Pakis, D. Waldram, \emph{%
$G$-structures and wrapped NS5-branes}, Commun. Math. Phys. \textbf{247}
(2004), 421--445. }

\bibitem{GMW} {\ J.P. Gauntlett, D. Martelli, D. Waldram, \emph{Superstrings
with intrinsic torsion}, Phys. Rev. \textbf{D 69} (2004), 086002, 27 pp. }

\bibitem{GKP} {\ S.B. Giddings, S. Kachru, J. Polchinski, \emph{Hierarchies
from fluxes in string compactifications}, Phys. Rev. \textbf{D 66} (2002), 106006, 16 pp. }

\bibitem{GPap} {\ J. Gillard, G. Papadopoulos, D. Tsimpis, \emph{Anomaly,
fluxes and $(2,0)$ heterotic-string compactifications}, JHEP \textbf{0306} (2003), 035. }

\bibitem{GP} {\ E. Goldstein, S. Prokushkin, \emph{Geometric model for
complex non-K\"ahler manifolds with $SU(3)$ structure}, Commun. Math. Phys.
\textbf{251} (2004), 65--78. }

\bibitem{GLP} {\ U. Gran, P. Lohrmann, G. Papadopoulos, \emph{The spinorial
geometry of supersymmetric heterotic string backgrounds}, JHEP \textbf{0602} (2006),
063. }

\bibitem{GPRS} {\ U. Gran, G. Papadopoulos, D. Roest, P. Sloane, \emph{%
Geometry of all supersymmetric type I backgrounds}, JHEP \textbf{0708} (2007), 074. }

\bibitem{GPR} {\ U. Gran, G. Papadopoulos, D. Roest, \emph{Supersymmetric
heterotic string backgrounds}, Phys. Lett. \textbf{B 656} (2007), 119--126. }

\bibitem{P} {\ U. Gran, G. Papadopoulos, \emph{Solution of heterotic Killing
spinor equations and special geometry}, Special metrics and supersymmetry, 144--161,
AIP Conf. Proc., 1093, Amer. Inst. Phys., Melville, NY, 2009. }

\bibitem{HK} {\ A. Hanany, B. Kol, \emph{On orientifolds, discrete torsion,
branes and M-theory }, JHEP \textbf{0006} (2000), 013. }

\bibitem{HP} {\ A. Hanany, B. Pioline, \emph{(Anti-)instantons and the
Atiyah-Hitchin manifold}, JHEP \textbf{0007} (2000), 001. }

\bibitem{HP1} {\ P.S. Howe, G. Papadopoulos, \emph{Ultraviolet behavior of
two-dimensional supersymmetric non-linear sigma models}, Nuclear Phys. \textbf{%
B 289} (1987), 264--276. }

\bibitem{Hu86} {\ C.M. Hull, \emph{Compactifications of the heterotic
superstring}, Phys. Lett. \textbf{B 178} (1986), 357--364. }

\bibitem{Hull} {\ C.M. Hull, \emph{Anomalies, ambiguities and superstrings},
Phys. Lett. \textbf{B 167} (1986), 51--55. }

\bibitem{HT} {\ C.M. Hull, P.K. Townsend, \emph{The two loop beta function
for sigma models with torsion}, Phys. Lett. \textbf{B 191} (1987), 115--121. }

\bibitem{HuW} {\ C.M. Hull, E. Witten, \emph{Supersymmetric sigma models and
the heterotic string}, Phys. Lett. \textbf{B 160} (1985), 398--402. }

\bibitem{IMY} {\ H. Imazato, S. Mizoguchi, M. Yata, \emph{Taub-NUT crystal},
Int. J. Mod. Phys. A \textbf{26} (2011), 5143--5169. }

\bibitem{Iv0} {\ S. Ivanov, \emph{Heterotic supersymmetry, anomaly
cancellation and equations of motion}, Phys. Lett. \textbf{B 685} (2010), 190--196. }

\bibitem{II} {\ P. Ivanov, S. Ivanov, \emph{$SU(3)$-instantons and $%
G_2,Spin(7) $-Heterotic string solitons}, Commun. Math. Phys. \textbf{259}
(2005), 79--102. }

\bibitem{IP1} {\ S. Ivanov, G. Papadopoulos, \emph{Vanishing theorems and
string backgrounds}, Class. Quantum Grav. \textbf{18} (2001), 1089--1110. }

\bibitem{IP2} {\ S. Ivanov, G. Papadopoulos, \emph{A no-go theorem for
string warped compactifications}, Phys. Lett. \textbf{B 497} (2001), 309--316. }

\bibitem{KY} {\ T. Kimura, P. Yi, \emph{Comments on heterotic flux
compactifications}, JHEP \textbf{0607} (2006), 030. }

\bibitem{KM} {\ T. Kimura, S. Mizoguchi, \emph{Chiral generations on
intersecting 5-branes in heterotic string theory }, JHEP \textbf{1004} (2010), 028. }

\bibitem{y1} {\ J. Li, S.-T. Yau, \emph{The existence of supersymmetric
string theory with torsion}, J. Diff. Geom. \textbf{70}
(2005), 143--181. }

\bibitem{MS} {\ D. Martelli, J. Sparks, \emph{Non-K\"ahler heterotic rotations}%
, Adv. Theor. Math. Phys. \textbf{15} (2011), 131--174. }

\bibitem{MY} {\ S. Mizoguchi, M. Yata, \emph{Family unification via
quasi-Nambu-Goldstone fermions in string theory}, Prog. Theor. Exp. Phys.
(2013), 053B01. }

\bibitem{Pap} {\ G. Papadopoulos, \emph{New half supersymmetric solutions of
the heterotic string},  Class. Quantum Grav. \textbf{26} (2009) 135001, 26 pp.}

\bibitem{Sen} {\ A. Sen, \emph{$(2,0)$ supersymmetry and space-time
supersymmetry in the heterotic string theory}, Nuclear Phys. \textbf{B 278}
(1986), 289--308. }

\bibitem{Sen0} {\ A. Sen, \emph{A note on enhanced gauge symmetries in M-
and string theory}, JHEP \textbf{9709} (1997), 001. }

\bibitem{Str1} {\ A. Strominger, \emph{Heterotic solitons}, Nuclear Phys.
\textbf{B 343} (1990) 167--184. [Erratum-ibid. \textbf{B 353} (1991) 565.] }

\bibitem{Str} {\ A. Strominger, \emph{Superstrings with torsion},
Nuclear Phys. \textbf{B 274} (1986), 253--284. }

\bibitem{U-Y} {\ K. Uhlenbeck, S.-T. Yau, \emph{On the existence of
Hermitian-Yang-Mills connections in stable vector bundles}, Comm. Pure Appl.
Math. \textbf{39} (1986), no. S, suppl., S257--S293. }

\bibitem{UV1} {\ L. Ugarte, R. Villacampa, \emph{Non-nilpotent complex geometry
of nilmanifolds and heterotic supersymmetry}, Asian J. Math. (to appear), arXiv:0912.5110 [math.DG]. }

\bibitem{UV2} {\ L. Ugarte, R. Villacampa, \emph{Balanced Hermitian geometry on
6-dimensional nilmanifolds}, Forum Math. (to appear), arXiv:1104.5524v2 [math.DG]. }

\bibitem{Yau} {\ S.-T. Yau, \emph{On the Ricci curvature of a compact K\"ahler
manifold and the complex Monge-Amp\`ere equation, I}, Comm. Pure Appl. Math.
\textbf{31} (1978), 339--411. }

\end{thebibliography}
\end{document}